\documentclass[12pt,a4paper]{paper}
\usepackage[utf8]{inputenc}

\usepackage{amsmath}
\usepackage{amsfonts}
\usepackage{amssymb}
\usepackage{graphicx}
\usepackage{amsthm} 
\usepackage{float}
\usepackage{fancyhdr}
\usepackage{enumitem}
\usepackage{color}

\usepackage[a4paper, hmargin=2.75cm, vmargin={4cm, 4cm}]{geometry}

\usepackage[hidelinks]{hyperref}

\usepackage[nobysame]{amsrefs}

\newtheorem{theorem}{Theorem}
\newtheorem{theorema}{Theorem}
\newtheorem{lemma}{Lemma}

\newtheorem{claim}{Claim}

\newcommand{\pr}{\mathbb{P}}
\newcommand{\E}{\mathbb{E}}

\newcommand\inner[2]{\langle #1, #2 \rangle}

\newcommand{\RomanNumeralCaps}[1]
{\MakeUppercase{\romannumeral #1}}

\theoremstyle{definition}
\newtheorem{definition}{Definition}
\newtheorem*{remark}{Remark}

\title{Recovering the lattice from its random perturbations}
\author{Oren Yakir \thanks{Supported by ISF Grants 382/15, 1903/18 and by ERC Advanced Grant 692616} \\ {\small School of Mathematical Sciences, Department of Pure Mathematics, Tel Aviv University} \\ {\small Email: {\tt oren.yakir@gmail.com}}  }
\date{}
\begin{document}
	
	\maketitle
	
	\begin{abstract}
		Given a $d$-dimensional Euclidean lattice we consider the random set obtained by adding an independent Gaussian vector to each of the lattice points. In this note we provide a simple procedure that recovers the lattice from a single realization of the random set. 
	\end{abstract}
		
	\section{Introduction and the main result}
	
	Let $\mathcal{L}\subset \mathbb{R}^d$ be a $d$-dimensional lattice and let $\mathcal{D}$ be its fundamental domain. We assume that $m_{d}(\mathcal{D}) = 1$, where $m_d$ is the Lebesgue measure in $\mathbb{R}^d$. Let $\left\{\xi_n\right\}_{n\in\mathcal{L}}$ be independent and identically distributed random vectors in $\mathbb{R}^d$, all with common probability law $\xi$ and let $\left(\Omega,\mathcal{F},\pr\right)$ be the probability space on which they are defined. 
	
	We study the random point process $W= W(\mathcal{L},\xi)$ given by
	\begin{equation}
		\label{eq:definition_of_the_process}
		W := \left\{n+\xi_n \mid n\in\mathcal{L} \right\}
	\end{equation}
	and address the following recovery problem: \textit{Given a realization of the random set $W$, is it possible to determine (with probability one) what is the underlying lattice $\mathcal{L}$?} To formulate our result, we use the standard notation $e(t):=\exp\left(2\pi i t\right)$. We consider the random exponential sum
	\begin{equation}
		\label{eq:fourier_transform_of_random_averging_measure}
		M_R (\lambda) := \frac{1}{m_d(B_R)} \sum_{w\in W\cap B_R} \overline{e(\inner{w}{\lambda})},
	\end{equation}
	where $B_{R}:=\left\{|x|\leq R\right\}$. 
	Recall that the \emph{dual lattice} to $\mathcal{L}$ is given by
	\[
		\mathcal{L}^\ast := \left\{m\in \mathbb{R}^d \mid \forall n\in \mathcal{L}, \ \inner{n}{m}\in \mathbb{Z} \right\}.
	\]
	Then $\mathcal{L}^\ast$ is also a lattice and $\mathcal{L} = \left(\mathcal{L}^\ast\right)^\ast$.
	
	\begin{theorem}
		\label{thm:random_psf}
		Suppose that $W$ and $M_R$ are given by (\ref{eq:definition_of_the_process}) and (\ref{eq:fourier_transform_of_random_averging_measure}). Assume that there exist some $\varepsilon>0$ such that
		\begin{equation}
		\label{eq:moment_condition_for_xi}
		\E\left[|\xi|^{d+\varepsilon}\right]<\infty.
		\end{equation} 
		Then, almost surely, for all $\lambda\in\mathbb{R}^d$, we have
		\begin{equation}
		\label{eq:random_psf}
		\lim_{R\to\infty} M_R (\lambda) =  \begin{cases}
		\varphi_\xi(\lambda) & \lambda \in \mathcal{L}^{\ast}, \\ 0 & \lambda\not\in \mathcal{L}^\ast,
		\end{cases}
		\end{equation}
		where $\varphi_\xi(\lambda) := \E\left[\overline{e(\inner{\xi}{\lambda})}\right]$ is the characteristic function of the random vector $\xi$.
	\end{theorem}
	\subsubsection*{Several Remarks on Theorem \ref*{thm:random_psf}}
	\begin{enumerate}[label=\bf{\arabic*}.]
		\item The moment condition \eqref{eq:moment_condition_for_xi} is probably not sharp for \eqref{eq:random_psf} to hold, and is merely of technical convenience. Still, in the view of Lemma \ref{lemma:boundry_bound_BC} (see Section \ref{sec:fourier_averging_random_set}), some moment condition should be expected. Throughout the paper we will assume that \eqref{eq:moment_condition_for_xi} holds without stating it explicitly in the different results. We further note that the moment condition \eqref{eq:moment_condition_for_xi} guarantees that the sum in \eqref{eq:fourier_transform_of_random_averging_measure} is almost surely finite and hence the function $M_{R}(\lambda)$ is well defined.
		
		\item If $\xi$ is such that $\varphi_\xi(\lambda)\not= 0$ for all $\lambda \in \mathbb{R}^d$ (in particular, if $\xi$ is Gaussian), then Theorem \ref{thm:random_psf} gives rise to a procedure that almost surely recovers $\mathcal{L}$. More accurately, let $\mathcal{B}(\mathbb{R}^d)$ be the Borel sigma-algebra on $\mathbb{R}^d$ and let $\Theta$ be the set of all locally finite (i.e. finite intersection with compacts) subsets of $\mathbb{R}^d$. Endow $\Theta$ with $\mathcal{G}$ which is the smallest sigma-algebra such that all maps
		\[
		n_\theta: \mathcal{B}(\mathbb{R}^d) \to \mathbb{Z}_{\ge 0} \cup \{\infty\}, \quad  n_\theta(B):= \#\left(\theta\cap B\right), 
		\]
		are measurable for all $\theta \in \Theta$. $\left(\Theta,\mathcal{G}\right)$ is a measurable space and in fact
		\[
		W: \Omega \to \Theta 
		\] 
		is a measurable map. By considering the map $T:\Omega\times \mathbb{R}^d\to \mathbb{C}$ given by $$T(\omega,\lambda):=\limsup_{R\rightarrow\infty}M_R(\lambda)$$ 
		we conclude from Theorem \ref{thm:random_psf} that if $\varphi_\xi(\lambda) \not=0$ for all $\lambda$ then
		\begin{equation*}
		\label{eq:recovering_the_lattice_if_non_vanishing_char}
		\pr\left(\left\{\lambda \in \mathbb{R}^d \mid T\left(\omega,\lambda\right) \not= 0 \right\} = \mathcal{L}^\ast\right) = 1.
		\end{equation*}
		We end the introduction with some simulations that demonstrates this recovery method.
		
		\item We do not assume in Theorem \ref{thm:random_psf} that $\xi$ has zero expectation. This means that we can also recover $\mathcal{L}$ from a random set of the form
		\[
		\widetilde{W}:=\left\{n + c + \xi_n \mid n\in \mathcal{L} \right\}
		\]
		where $\{\xi_n\}$ are i.i.d. random vectors and $c\in \mathbb{R}^d$ is an arbitrary (non-random) vector . The only difference is that the limiting function in (\ref{eq:random_psf}) is multiplied by $e(-\inner{\lambda}{c})$.
		
		\item The normalization assumption $m_d(\mathcal{D}) = 1$ is not essential. It would be clear from the proof that if we do not normalize the the limiting function in (\ref{eq:random_psf}) is multiplied by $(m_d(\mathcal{D}))^{-1}$.
		
	\end{enumerate}
	
	Notice that a simple application of Birkhoff ergodic theorem combined with Fubini theorem implies that for each $\lambda\in\mathbb{R}^d$ there exist an event $E_\lambda \in \mathcal{F}$ with $\pr(E_\lambda) = 0$ such that relation (\ref{eq:random_psf}) holds for all $\omega\in \Omega\setminus E_\lambda$. The point of Theorem \ref{thm:random_psf} is that we may choose a single event $E\in \mathcal{F}$, $\pr(E) = 0$, such that for all $\omega\in \Omega\setminus E$ relation (\ref{eq:random_psf}) holds for all $\lambda\in \mathbb{R}^d$ at the same time. This type of ``uniformity" result is closely related to the Wiener-Wintner theorem, first appearing in the celebrated paper \cite{wiener_wintner}. 
	
	\begin{theorema}[{\cite{wiener_wintner}}]
		\label{thm:wiener_wintner_thm}
		Suppose that $\tau$ is a measure-preserving transformation of a measure space $S$ with finite measure. If $f$ is a complex-valued integrable function on $S$ then there exists a measure zero set $E$ such that the limit
		\[
		\lim_{N \to \infty} \frac{1}{2N+1} \sum_{j=-N}^{N}e^{ijt} f(\tau^{j}s)
		\] 
		exists for all real $t$ and for all $s\not\in E$.
	\end{theorema}
	
	 The case $t=0$ in Theorem \ref{thm:wiener_wintner_thm} is essentially the Birkhoff ergodic theorem. For (many) different proofs of Theorem \ref{thm:wiener_wintner_thm} and possible extensions see the book \cite{assani}. Although we do not use directly the Wiener-Wintner theorem in our proof of Theorem \ref{thm:random_psf}, the connection is evident. In particular, a key step towards proving Theorem \ref{thm:random_psf} is to introduce the notion of sequences having correlations (to be defined later) and study their spectral properties. This notion was originally introduced by N. Wiener \cite{wiener}*{Chapter~\RomanNumeralCaps{4}}.
	 
	 Considering random displacements of given lattice points is a natural model which appears in several physical context. Probably the most well-known example is thermal motions in the Einstein approximation of a solid. We refer the reader to \cite[Section~5]{baake_grimm} for a survey of previous results on the perturbed lattice from the mathematical physics point of view. We also mention a result by Hof \cite{hof}, where the diffraction of the random measure associated with the set $W$ was computed, and as in our work, the self averaging of the infinite system implies almost sure results. Another nice instance in the physics literature appears in the work of Gabrielli, Joyce and Sylos Labini \cite{gabrielli_joyce_syloslab} (see also the book \cite{gabrielli_joyce_syloslab_pietroneor}), where independent displacements of the lattice points appears as a cosmological model (see section ``the shuffled lattice" therein).
	
	Mathematically, random perturbations of lattice points is a natural example of a super-homogeneous point process. That is, random point sets where the variance of the number of points in a domain $V$ grows slower than the volume of $V$, see \cite{ghosh_lebowitz}. The notion of super-homogeneous point process (which are sometimes called \emph{hyperuniform}) was introduced by Stillinger and Torquato in \cite{torquato_stillinger}. Sodin and Tsirelson \cite{sodin_tsirelson} considered Gaussian perturbations of the lattice points as a toy-model for the more involved super-homogeneous point process, obtained by considering the zero set of an analytic function whose Taylor coefficients are independent (complex) Gaussian random variables. In the context of recovery problems, Peres and Sly \cite{peres_sly} proved that given we know the underlying lattice $\mathcal{L}$, the problem of detecting whether or not a point was deleted from $W$ is much less tractable. In particular, they proved that if $\xi_n$ are mean-zero Gaussian random vectors with independent components, each of variance $\sigma^2$, then for $d\ge 3$ and $\sigma=\sigma(d)$ large enough it is impossible to detect whether a point was deleted, while for small $\sigma$ such a detection is possible. In the work \cite{yakir}, we study the mean and fluctuations of linear statistics of the point process $W$ in the case of Gaussian perturbations.
	
	After the completion of this paper, we have learned of the recent work by Klatt, Kim and Torquato \cite{klatt_kim_torquato} which is closely related to our work, and in some sense, complements it. There, a formula similar to \eqref{eq:random_psf} is derived at the level of expectations. The main concern in \cite{klatt_kim_torquato} is when $\varphi_{\xi}$ vanishes on the dual lattice $\mathcal{L}^\ast$ (in their terminology, when the process is \emph{cloaked}) and comparing different metrics which measure this ``vanishing" of periodicity.
	
	\subsection{Simulations}
	We present here some numerical simulations to illustrate the recovery method we suggest. A natural example to consider is when the perturbations of the lattice points are (symmetric) bivariate Gaussian vectors with a dispersion parameter $a>0$, i.e,
	
	\begin{equation}
	\label{eq:gaussian_density}
	{\rm d} \xi(x) = \frac{1}{\pi a} \exp(-|x|^2/a)\cdot {\rm d} m_2(x) .
	\end{equation} 
	The characteristic function in this case is given by $\varphi_{\xi}(\lambda) = \exp(-a\pi^2|\lambda|^2)$. We will work with the following lattices in $\mathbb{R}^2$ given by
	\[
	\mathcal{L}_{1} := \mathbb{Z}^2, \qquad \mathcal{L}_{2}:= A\cdot\mathbb{Z}^2, \qquad 
	\text{where } A= \begin{pmatrix}
	2 & 1/2 \\ 0 & 1/2
	\end{pmatrix}.
	\]
	We generate two (independent) processes $W_1$ and $W_2$ as follows,
	\begin{equation*}
	W_j:= \left\{n+ \xi_n^j \mid n\in \mathcal{L}_j \right\} \qquad \text{for } j=1,2,
	\end{equation*}
	where $\{\xi_n^1\}$ and $\{\xi_n^2\}$ are independent bivariate Gaussian random vectors given by (\ref{eq:gaussian_density}). 
	
	\begin{figure}[H]
		\begin{center}
			\scalebox{1}{\includegraphics{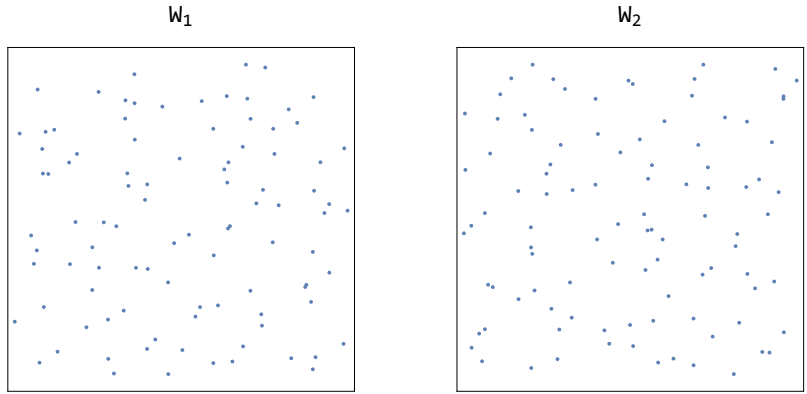}}
			\caption{Realizations of the process $W_1$ and $W_2$ in the box $[-5,5]^2$ with $a=0.1$.}
			\label{fig:realizations_of_W_j}
		\end{center}
	\end{figure}

	Let $M_R^{j}(\lambda)$ be the random exponential sum (\ref{eq:fourier_transform_of_random_averging_measure}), which corresponds to the different $W_j$'s. We will consider the random sets given by
	\begin{align*}
		X_{R,\beta}^{j} := \left\{\lambda\in \mathbb{R}^2 \mathrel{\bigg|} \frac{1}{\pi R^2} \sum_{w\in W_j\cap B_R} \cos\left(2\pi \inner{w}{\lambda} \right) > \beta  \right\} = \left\{\lambda\in \mathbb{R}^2 \mid \text{Re}\left( M_R^{j}(\lambda)\right) > \beta \right\}
	\end{align*}
	
	\begin{figure}[H]
		\begin{center}
			\scalebox{1}{\includegraphics{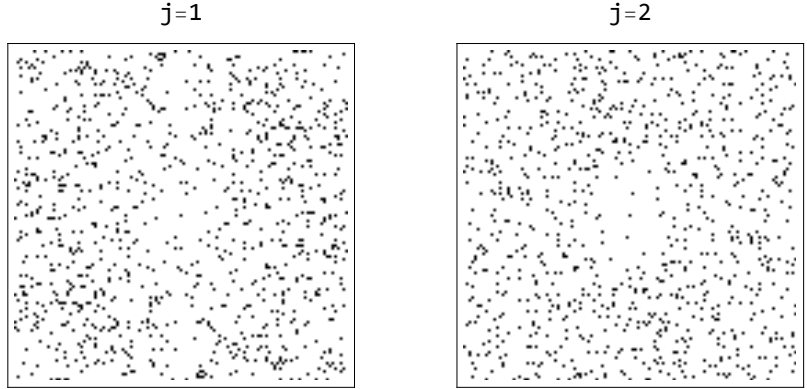}}
			\caption{Realizations of the set $X_{R,\beta}^{j}$ with $a=0.1$, $\beta=0.007$ and $R=100$.}
			\label{fig:realizations_of_X_100_beta_j}
		\end{center}
	\end{figure}

	for $j=1,2$ and $\beta\in(0,1)$. Theorem \ref{thm:random_psf} asserts that, almost surely, the set $X_{R,\beta}^{j}$ converges as $R\to\infty$ to the set of all points $\lambda\in \mathcal{L}_j^{\ast}$ such that $\text{Re}\left(\varphi_\xi(\lambda)\right)>\beta$. In the Gaussian case this is simply the set
	\begin{equation*}
	\label{eq:recovering_the_lattice_gaussian}
	\mathcal{L}_{j}^{\ast} \cap \left\{\lambda : |\lambda|< \sqrt{-\log \beta/a\pi^2} \right\}. 
	\end{equation*}

	In Figures \ref{fig:realizations_of_X_100_beta_j} and \ref{fig:realizations_of_X_250_beta_j} we demonstrate this convergence by plotting the set $X_{R,\beta}^{j}$ with the same realizations of the $W_j$'s from Figure \ref{fig:realizations_of_W_j}.
	
	\begin{figure}[H]
		\begin{center}
			\scalebox{1}{\includegraphics{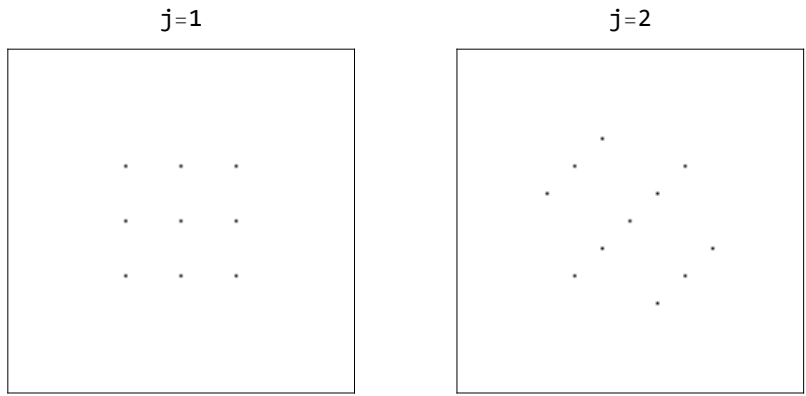}}
			\caption{Realizations of the set $X_{R,\beta}^{j}$ with $a=0.1$, $\beta=0.007$ and $R=250$.}
			\label{fig:realizations_of_X_250_beta_j}
		\end{center}
	\end{figure}
	
	\subsubsection*{Acknowledgments}
	
	I am deeply grateful to my advisors, Alon Nishry and Mikhail Sodin, for their attentive guidance, constant encouragement and for many stimulating discussions. I also thank the anonymous referee for spotting a gap in Lemma \ref{lemma:boundry_bound_BC} and for many helpful comments. 
	
	\section{Wiener sequences and their properties}
	
	Let $u:\mathcal{L}\to \mathbb{C}$ be a sequence indexed by a given lattice $\mathcal{L}$. We start with some classical results regarding sequences having correlations which we will use later when proving Theorem \ref{thm:random_psf}. We refer the interested reader to the book by Queff{\'e}lec \cite[Chapter~4]{queffelec} for a more elaborate introduction to the theory of Wiener sequences. 
	
	\begin{definition}
		We say that $u:\mathcal{L}\to \mathbb{C}$ is a \emph{Wiener sequence} if for any $k\in\mathcal{L}$ there exists the limit
		\begin{equation}
		\label{eq:correlation_of_seq_u}
		\gamma_{u}(k):= \lim_{R\to\infty} \frac{1}{m_d\left(B_{R}\right)} \sum_{n\in\mathcal{L}\cap B_R} u(n)\overline{u(n+k)}.
		\end{equation}
	\end{definition}
	We call $\gamma_{u}:\mathcal{L}\to \mathbb{C}$ the \emph{correlation sequence} of the Wiener sequence $u$. Any Wiener sequence $u$ gives rise to a (canonical) Borel measure supported on $\mathcal{D}$ which we denote by $\mu_u$. It is constructed as follows. For any $k_j,k_\ell \in \mathcal{L}$ we have
	\begin{align*}
	\gamma_{u}\left(k_\ell - k_{j}\right) &= \lim_{R\rightarrow\infty} \frac{1}{m_d\left(B_{R}\right)} \sum_{n\in\mathcal{L}\cap B_R} u(n)\overline{u(n+k_\ell-k_j)} \\ &= \lim_{R\rightarrow\infty} \frac{1}{m_d\left(B_{R}\right)} \sum_{n\in\mathcal{L}\cap B_R} u(n-k_\ell)\overline{u(n-k_j)},
	\end{align*}
	For all complex numbers $c_0,\ldots,c_m$ we obtain
	\begin{align*}
	\sum_{0\leq \ell,j \leq m} c_\ell \overline{c_{j}} \gamma_{u}\left(k_\ell - k_{j}\right) & = \lim_{R\to\infty} \frac{1}{m_d(B_R)}\sum_{n\in\mathcal{L}\cap B_R} \left(\sum_{0\leq \ell,j \leq m} c_\ell \overline{c_{j}} u(n-k_\ell)\overline{u(n-k_j)}\right) \\ &= \lim_{R\to\infty}  \frac{1}{m_d(B_R)} \sum_{n\in\mathcal{L}\cap B_R} \left| \sum_{\ell=1}^{m} c_\ell u(n-k_\ell) \right|^2 \geq 0.
	\end{align*}
	Hence, by the Bochner-Herglotz theorem, there exist a unique positive Borel measure $\mu_u$ on $\mathcal{D}$ (the dual group to $\mathcal{L}$) such that $\widehat{\mu}_{u}(k) = \gamma_u(k)$ for all $k\in \mathcal{L}$. We will refer to $\mu_u$ as the \emph{spectral measure} of the Wiener sequence $u$. The following lemma will be key in the proof of Theorem \ref{thm:random_psf}. 
	
	\begin{lemma}
		\label{lemma:upper_bound_for_mix_spectral_measures_ac}
		Let $u,v$ be Wiener sequences and let $\mu_u,\mu_v$ be their spectral measures, respectively. Suppose that $\mu_u$ and $\mu_v$ are mutually singular, then
		\[
		\lim_{R\rightarrow\infty} \frac{1}{m_d(B_R)} \sum_{n\in  \mathcal{L}\cap B_R} u(n)\overline{v(n)} = 0.
		\]
	\end{lemma}
	
	Lemma \ref{lemma:upper_bound_for_mix_spectral_measures_ac} goes back to Coquet, Kamae and Mend\`{e}s France \cite[Theorem~2]{coquet_kamae_menfrance}, which proved this lemma for the case $\mathcal{L} = \mathbb{Z}$ and $d=1$. For the convenience of the reader, we add a proof of Lemma \ref{lemma:upper_bound_for_mix_spectral_measures_ac} for the general case as stated in appendix \ref{appendix:proof_of_lemma_upper_bound}.
	
	\section{Fourier Averaging of the Random Set}
	\label{sec:fourier_averging_random_set}
	Let $N_{\mathcal{L}}(R) :=\#\left(\mathcal{L}\cap B_R\right)$ be the number of lattice points which fall inside a ball of radius $R\geq 1$ centered at the origin. The classical Gauss-type bounds yields that there exists a constant $C=C(\mathcal{L})>0$ so that
	\begin{equation}
	\label{eq:lattice_points_in_circle_bound}
	\left|N_\mathcal{L}(B_R) - m_d(B_{R})\right| \leq C R^{d-1}
	\end{equation}
	for all $R\ge1$, see \cite[Proposition~1]{baake_moody_pleasants}.
	\begin{lemma}
		\label{lemma:boundry_bound_BC}
		We have that, almost surely,
		\begin{equation*}
		\label{eq:boundry_bound_BC}
		\lim_{R\to\infty} \frac{1}{R^d}\#\left\{n\in \mathcal{L} : |n|\leq R,\ |n+\xi_n|>R \right\} = 0, 
 		\end{equation*}
 		and
 		\begin{equation*}
 			\lim_{R\to\infty} \frac{1}{R^d}\#\left\{n\in \mathcal{L} : |n|> R,\ |n+\xi_n|\leq R \right\} = 0.
 		\end{equation*}
	\end{lemma}
	\begin{proof}
		Let $\delta:=\varepsilon/2(d+\varepsilon)>0$ with the same $\varepsilon\in(0,1)$ as in (\ref{eq:moment_condition_for_xi}). By Chebyshev's inequality
		\begin{equation*}
		\pr\left(|\xi_n|\geq |n|^{1-\delta}\right)  \leq \frac{\E\left[|\xi_n|^{d+\varepsilon}\right]}{|n|^{(d+\varepsilon)(1-\delta)}} = \frac{\E\left[|\xi_n|^{d+\varepsilon}\right]}{|n|^{d+\varepsilon/2}}
		\end{equation*}
		for any $n\in\mathcal{L}\setminus \{0\}$. Therefore
		\[
		\sum_{n\in\mathcal{L}} \pr\left(|\xi_n|\geq |n|^{1-\delta}\right) \leq 1 + \E\left[|\xi|^{d+\varepsilon}\right] \sum_{n\in \mathcal{L}\setminus\{0\} } |n|^{-d-\varepsilon/2} <\infty. 
		\]
		Hence, by the Borel-Cantelli lemma, the random variable
		\begin{equation}
		\label{eq:BC__points_dont_go_far}
		X:= \#\left\{n\in\mathcal{L} : |\xi_n|\geq |n|^{1-\delta} \right\}
		\end{equation}
		is almost surely finite. By the triangle inequality
		\begin{align}
			\label{eq:triangle_inequality_boundary_bc}
			\#&\left\{n\in \mathcal{L} : |n|\leq R,\ |n+\xi_n|>R \right\} \nonumber \\ & \leq \#\left\{n\in \mathcal{L} : |n|\leq R,\ |n+\xi_n|>R , \ |\xi_n| < |n|^{1-\delta} \right\} + X \nonumber \\ & \leq \#\left\{n\in \mathcal{L} : |n|\leq R,\ |n| + |n|^{1-\delta} > R \right\} + X.
		\end{align}
		Furthermore, \eqref{eq:lattice_points_in_circle_bound} tells us that
		\begin{equation*}
			 \#\left\{n\in \mathcal{L} : |n|\leq R,\ |n| + |n|^{1-\delta} > R \right\} \leq N_{\mathcal{L}} (B_R) - N_{\mathcal{L}} (B_{R-R^{1-\delta}}) \leq C R^{d-\delta}
		\end{equation*}
		for some $C= C(\mathcal{L})>0$. Plugging the above into \eqref{eq:triangle_inequality_boundary_bc} we see that
		\[
		\#\left\{n\in \mathcal{L} : |n|\leq R,\ |n+\xi_n|>R \right\} \leq C R^{d-\delta} + X
		\] 
		which gives the first displayed formula.
		
		For the second displayed formula, we use \eqref{eq:lattice_points_in_circle_bound} once more and see that
		\begin{equation*}
			\#\left\{n\in \mathcal{L} : R < |n|\leq R + R^{1-\delta/2} \right\} \leq C R^{d-\delta/2}
		\end{equation*}
		so the proof will follow once we show that 
		\begin{equation}
			\label{eq:BC_inequality_for_points_far_away}
			\#\left\{n\in \mathcal{L} : |n|> R + R^{1-\delta/2},\ |n + \xi_n |\leq R \right\} \leq X.
		\end{equation}
		Indeed, let $|n|> R + R^{1-\delta/2}$ be such that $|\xi_n|\leq |n|^{1-\delta}$. We have
		\begin{equation*}
			|n+\xi_n| \ge |n| - |\xi_n| \ge R + \left[|n|- R - |n|^{1-\delta}\right] \ge R + c(\delta) R^{1-\delta/2}
		\end{equation*}
		for some constant $c(\delta)>0$. This gives \eqref{eq:BC_inequality_for_points_far_away} and the proof of Lemma \ref{lemma:boundry_bound_BC} follows.
	\end{proof}
	\begin{lemma}
		\label{lemma:main_term_ww}
		Almost surely, for all $\lambda\in\mathbb{R}^d$, we have
		\begin{equation*}
		\lim_{R\to\infty} \frac{1}{m_d(B_R)} \sum_{n\in \mathcal{L}\cap B_R} \overline{e(\inner{n}{\lambda})}\left\{\overline{e(\inner{\xi_n}{\lambda})} - \varphi_{\xi}(\lambda)\right\} = 0.
		\end{equation*}
	\end{lemma}
	We end this section by showing how Lemma \ref{lemma:main_term_ww} implies Theorem \ref{thm:random_psf} and devote the next section to the proof of Lemma \ref{lemma:main_term_ww}.
	\begin{proof}[Proof of Theorem \ref{thm:random_psf}]
		Since
		\[
		\lim_{R\rightarrow\infty}\frac{1}{m_d(B_R)}\sum_{n\in \mathcal{L}\cap B_R} \overline{e(\inner{n}{\lambda})} = \begin{cases}
		1 & \lambda\in\mathcal{L}^\ast, \\ 0  & \lambda\not\in\mathcal{L}^\ast,
		\end{cases}
		\]
		we conclude from Lemma \ref{lemma:main_term_ww} that, almost surely, for all $\lambda\in\mathbb{R}^d$,
		\begin{equation}
		\label{eq:cases_for_random_function_limit}
		\lim_{R\rightarrow\infty}\frac{1}{m_d(B_R)}\sum_{n\in  \mathcal{L}\cap B_R} \overline{e(\inner{n+\xi_n}{\lambda})} = \begin{cases}
		\varphi_\xi(\lambda) & \lambda\in\mathcal{L}^\ast, \\ 0  & \lambda\not\in\mathcal{L}^\ast.
		\end{cases}
		\end{equation}
		By Lemma \ref{lemma:boundry_bound_BC},
		\begin{align*}
		\sup_{\lambda\in\mathbb{R}^d}&\left|M_R(\lambda) -\frac{1}{m_d(B_R)}\sum_{n\in  \mathcal{L}\cap B_R} \overline{e(\inner{n+\xi_n}{\lambda})}\right|  \\ &= \frac{1}{m_d(B_R)}\sup_{\lambda\in\mathbb{R}^d}\left|\sum_{\substack{n\in \mathcal{L} \\ |n+\xi_n|\leq R}} \overline{e(\inner{n+\xi_n}{\lambda})} - \sum_{n\in  \mathcal{L}\cap B_R} \overline{e(\inner{n+\xi_n}{\lambda})}\right|\\ &\leq \frac{\#\left\{n\in \mathcal{L} : |n|\leq R,\ |n+\xi_n|>R \right\}}{m_d(B_R)} \\ &\qquad + \frac{\#\left\{n\in \mathcal{L} : |n|> R,\ |n+\xi_n|\leq R \right\}}{m_d(B_R)} \xrightarrow{R\to\infty} 0,
		\end{align*}
		almost surely. Combining with relation (\ref{eq:cases_for_random_function_limit}), we finish the proof.
	\end{proof}
	
	\subsection{Proof of Lemma \ref{lemma:main_term_ww}}
	We now turn to the proof of Lemma \ref{lemma:main_term_ww}. Let $A_\lambda(n):= \overline{e(\inner{\xi_n}{\lambda})} - \varphi_{\xi}(\lambda)$. Notice that a simple application of Birkhoff ergodic theorem \cite[Theorem~16.2]{koralov_sinai} combined with the Fubini theorem implies that, for any fixed $\lambda\in \mathbb{R}^d$, we almost surely have that
	\begin{equation}
	\label{eq:main_term_ww}
	\lim_{R\rightarrow\infty} \frac{1}{m_d(B_R)} \sum_{n\in \mathcal{L}\cap B_R} A_\lambda(n)\overline{e(\inner{n}{\lambda})} = 0
	\end{equation}
	As we have explained in the introduction, the main point of Lemma \ref{lemma:main_term_ww} is that \textit{we may choose a single event which is independent of the $\lambda$'s}. The proof of Lemma \ref{lemma:main_term_ww} is inspired by ideas from a paper by Bellow and Losert \cite{bellow_losert}, where (among other things) an alternative proof of the Wiener-Wintner theorem is provided.
	\begin{claim}
		\label{claim:independent_random_variables_are_wiener}
		For every fixed $\lambda\in\mathbb{R}^d$, the sequence $\left\{A_\lambda(n)\right\}_{n\in\mathcal{L}}$ is, almost surely, a Wiener sequence with correlations given by
		\begin{equation}
		\label{eq:correlations_of_A_lambda}
		\gamma_{A_\lambda}(k) = \begin{cases}
		\E\left|\overline{e(\inner{\xi}{\lambda})} - \varphi_{\xi}(\lambda)\right|^2 & k=0, \\ 0 & k\in\mathcal{L}\setminus \{0\}.
		\end{cases}
		\end{equation}
	\end{claim}
	\begin{proof}
		Consider the random function $F_{R,k}:\mathbb{R}^d\to \mathbb{C}$ given by
		\begin{equation}
		\label{eq:definition_of_G_R_k}
		F_{R,k}(\lambda):= \frac{1}{m_d(B_R)} \sum_{n\in\mathcal{L}\cap B_R} A_{\lambda}(n)\overline{A_{\lambda}(n+k)}.
		\end{equation}
		We fix some $\lambda\in \mathbb{R}^d$ and turn to show that almost surely, $F_{R,k}(\lambda)\rightarrow \gamma_{A_\lambda}(k)$ as $R\to\infty$. For every fixed $k\in\mathcal{L}$ the sequence $$\left\{A_{\lambda}(n)\overline{A_{\lambda}(n+k)}\right\}_{n\in\mathcal{L}}$$ is ergodic with respect to the lattice shifts (in the sense defined in \cite[Chapter~16.3]{koralov_sinai}). By the Birkhoff ergodic theorem we see that
		\[
		\lim_{R\to\infty} F_{R,k}(\lambda) = \E\left[A_{\lambda}(0)\overline{A_{\lambda}(k)}\right]
		\]
		almost surely. The claim follows by observing that $A_{\lambda}(0)$ and $A_{\lambda}(k)$ are independent for all $k\not=0$.
	\end{proof}
	\begin{claim}
		\label{claim:derivative_bound_for_F}
		For every $k\in \mathcal{L}$ we have that, almost surely,
		\[
		\sup_{R\geq1} \sup_{\lambda\in\mathbb{R}^d} \left|\nabla F_{R,k}(\lambda)\right| <\infty.
		\]
	\end{claim}
	\begin{proof}
		Write $x=\left(x^1,\ldots,x^d\right)\in\mathbb{R}^d$ for a $d$-dimensional vector. Observe that
		\begin{align*}
		\nabla F_{R,k}(\lambda) &= \frac{1}{m_d(B_R)} \sum_{n\in\mathcal{L}\cap B_R} \nabla \left(A_{\lambda}(n)\overline{A_{\lambda}(n+k)}\right) \\ &= \frac{1}{m_d(B_R)} \sum_{n\in\mathcal{L}\cap B_R}  \overline{A_{\lambda}(n+k)} \nabla A_{\lambda}(n) + \frac{1}{m_d(B_R)} \sum_{n\in\mathcal{L}\cap B_R} A_{\lambda}(n) \nabla \overline{A_{\lambda}(n+k)}, 
		\end{align*}
		with
		\[
		\nabla A_{\lambda}(n) = -2\pi i \begin{pmatrix}
		\xi_{n}^{1}\overline{e(\inner{\xi_n}{\lambda})} - \E\left[\xi_{n}^{1}\overline{e(\inner{\xi_n}{\lambda})}\right] \\ \vdots \\ \xi_{n}^{d}\overline{e(\inner{\xi_n}{\lambda})} - \E\left[\xi_{n}^{d}\overline{e(\inner{\xi_n}{\lambda})}\right]
		\end{pmatrix}.
		\]
		Applying the triangle and Cauchy-Schwarz inequalities yields
		\begin{equation*}
		\left|\nabla A_{\lambda}(n)\right|^2 \lesssim \sum_{j=1}^{d} \left|\xi_n^j\overline{e(\inner{\xi_n}{\lambda})}\right|^2 + \sum_{j=1}^{d} \left|\E\left[\xi_{n}^{j}\overline{e(\inner{\xi_n}{\lambda})}\right]\right|^2 \lesssim \left|\xi_n\right|^2 + \E\left[\left|\xi_n\right|\right]^2,
		\end{equation*}
		and so, since $\left|A_n(\lambda)\right| \leq 2$, we obtain that
		\begin{align*}
		\sup_{\lambda\in\mathbb{R}^d} \left|\nabla F_{R,k}(\lambda)\right| \lesssim \frac{1}{m_d(B_R)} \sum_{n\in \mathcal{L}\cap B_R} \left(\left|\xi_n\right| + \left|\xi_{n+k}\right| + \E\left[|\xi|\right]\right).
		\end{align*}
		By the moment assumption (\ref{eq:moment_condition_for_xi}), we may apply the Strong Law of Large Numbers \cite[Theorem~7.2]{koralov_sinai} which yields that, almost surely,
		\[
		\lim_{R\to \infty} \frac{1}{m_d(B_R)} \sum_{n\in \mathcal{L}\cap B_R} \left(\left|\xi_n\right| + \left|\xi_{n+k}\right| + \E\left[|\xi|\right]\right) = 3\E\left[|\xi|\right].
		\]
		As every convergent sequence is bounded, we conclude that $\sup_{\lambda\in\mathbb{R}^d} \left|\nabla F_{R,k}(\lambda)\right|$ is bounded uniformly in $R\ge 1$ and hence the claim follows.
	\end{proof}
	\begin{proof}[Proof of Lemma \ref{lemma:main_term_ww}]
		Let $\Lambda$ be a countable dense subset of $\mathbb{R}^d$, and fix some $k\in\mathcal{L}$. Since a countable union of probability zero events is a probability zero event, we can conclude from Claim \ref{claim:independent_random_variables_are_wiener} that
		\begin{equation}
		\label{eq:limit_for_F_R_k}
		\lim_{R\to\infty} F_{R,k}(\lambda) = \gamma_{A_\lambda}(k)
		\end{equation}
		almost surely for all $\lambda\in \Lambda$. We now show that relation (\ref{eq:limit_for_F_R_k}) holds almost surely for all $\lambda\in\mathbb{R}^d$. 
		
		Indeed, for $\lambda\not\in \Lambda$, take a sequence $(\lambda_p)\subset \Lambda$ converging to $\lambda$ as $p\to\infty$, and denote for the moment $\gamma(\lambda):=\gamma_{A_\lambda}(k)$, where $\gamma_{A_\lambda}(k)$ is the same as in (\ref{eq:correlations_of_A_lambda}). By Claim \ref{claim:derivative_bound_for_F}, we almost surely have
		\begin{align*}
		\left|F_{R,k}(\lambda) -\gamma(\lambda) \right| & \leq \left|\gamma(\lambda_p)- \gamma(\lambda)\right| + \left|F_{R,k}(\lambda_p) - \gamma(\lambda_p)\right| + \left|F_{R,k}(\lambda) - F_{R,k}(\lambda_p)\right| \\ & \leq \left|\gamma(\lambda_p)- \gamma(\lambda)\right| + \left|F_{R,k}(\lambda_p) - \gamma(\lambda_p)\right| + M_k\left|\lambda - \lambda_p\right|,
		\end{align*}
		where,
		\[
		M_k:= \sup_{R\geq1} \sup_{\lambda\in\mathbb{R}^d} \left|\nabla F_{R,k}(\lambda)\right| < \infty.
		\]
		Since the function $\lambda \mapsto \gamma(\lambda)$ is continuous, the limit $R\to \infty$ followed by $p\to\infty$ yields that relation (\ref{eq:limit_for_F_R_k}) holds almost surely for all $\lambda\in \mathbb{R}^d$.
		
		Since the number of lattice points is countable, we conclude that, almost surely, $\left\{A_\lambda(n)\right\}_{n\in\mathcal{L}}$ is a Wiener sequence for all $\lambda\in\mathbb{R}^d$. The correlation sequence of $A_\lambda$ is the sequence $\gamma_{A_\lambda}$ defined in (\ref{eq:correlations_of_A_lambda}). The corresponding spectral measure is given by
		\begin{equation*}
		{\rm d}\mu_{A_\lambda}(x) = \E\left|\overline{e(\inner{\xi}{\lambda})} - \varphi_{\xi}(\lambda)\right|^2 {\rm d}m_d(x), 
		\end{equation*} 
		and is a constant multiple of Lebesgue measure on the fundamental domain $\mathcal{D}$. The sequence $\left\{e\left(\inner{\lambda}{n}\right)\right\}_{n\in\mathcal{L}}$ is also a Wiener sequence, with the correlation measure $\delta_{\lambda (\text{mod } \mathcal{L})}$, a point mass at the unique point in $\mathcal{D}$ given by $\lambda - n$ for some lattice point $n\in\mathcal{L}$. Clearly, $\delta_{\lambda (\text{mod } \mathcal{L})}$ is singular with respect to Lebesgue measure. Hence, we apply Lemma \ref{lemma:upper_bound_for_mix_spectral_measures_ac} and conclude that, almost surely, for all $\lambda\in\mathbb{R}^d$
		\[
		\lim_{R\rightarrow\infty} \frac{1}{m_d(B_R)} \sum_{n\in \mathcal{L} \cap B_R} A_n(\lambda)\overline{e\left(\inner{\lambda}{n}\right)} = 0
		\]
		which gives the desired result.
	\end{proof}	
	
	\begin{remark}
		Notice that we did not use the independence of $\xi_n$'s in a crucial way. The limit 
		\[
		\lim_{R\to\infty} F_{R,k}(\lambda) = \E\left[A_{\lambda}(0)\overline{A_{\lambda}(k)}\right]
		\]
		holds in a much more general setting and gives rise to spectral measures which are not necessarily a constant multiple of Lebesgue measure. Indeed, Theorem \ref{thm:random_psf} remains true if we assume that $\{\xi_n\}$ are only \emph{mixing} (in the sense of ergodic theory) with respect to the lattice shifts. For a precise definition of this notion see \cite[Section~16.3]{koralov_sinai}.
	\end{remark}
		
	\pagebreak
	
	\clearpage
	
	\appendix
	
	\section{Proof of Lemma \ref{lemma:upper_bound_for_mix_spectral_measures_ac}}
	\label{appendix:proof_of_lemma_upper_bound}
	
	Let $\mu$ and $\sigma$ be a finite Borel measures on $\mathcal{D}$, the fundamental domain to the lattice $\mathcal{L}$. We write $\mu \ll \sigma$ if $\mu$ is absolutely continuous with respect to $\sigma$.
	
	\begin{definition}
		\label{def:affiliation_of_measures}
		Let $\mu$ and $\nu$ be finite Borel measures on $\mathcal{D}$, and suppose $\sigma$ is another finite Borel measure such that $\mu \ll \sigma$ and $\nu \ll \sigma$. The \emph{affinity} between the measures $\mu$ and $\nu$ (sometimes called the \emph{Hellinger integral}) is defined as follows
		\begin{equation}
		\label{eq:affinity_of_measures}
		\rho\left(\mu,\nu\right) := \int_{\mathcal{D}} \left(\frac{{\rm d}\mu}{{\rm d}\sigma}\right)^{1/2} \left(\frac{{\rm d}\nu}{{\rm d}\sigma}\right)^{1/2} {\rm d} \sigma.
		\end{equation}
	\end{definition}
	We observe two properties which are immediate from (\ref{eq:affinity_of_measures}):
	\begin{itemize}
		\item $\rho(\mu,\nu)$ does not depend on the reference measure $\sigma$;
		\item $\rho(\mu,\nu) = 0$ if and only if $\mu$ and $\nu$ are mutually singular.
	\end{itemize}
	Recall that a family of positive measures $\left(\sigma_t\right)_{t>0}$ on $\mathcal{D}$ converges weak-star to a limiting measure $\sigma$ if for any bounded continuous function $f:\mathcal{D}\to \mathbb{R}$ 
	\[
	\int_{\mathcal{D}} f {\rm d} \sigma_t \longrightarrow \int_{\mathcal{D}} f{\rm d} \sigma
	\]
	as $t\to\infty$.  We will denote this convergence by $\sigma_t\xrightarrow{\ w^\ast\ } \sigma$.
	\begin{theorema}[{\cite[Theorem~2]{coquet_kamae_menfrance}}]
		\label{thm:weak_star_measures_bound_affiliation}
		Let $\left(\mu_t\right)$ and $\left(\nu_t\right)$ be two families of positive measures on $\mathcal{D}$ such that $\mu_t\xrightarrow{\ w^\ast\ }\mu$ and $\nu_t\xrightarrow{\ w^\ast\ }\nu$ as $t\to\infty$ for some finite measures $\mu$ and $\nu$. Then
		\begin{equation*}
		\label{eq:weak_star_measures_bound_affiliation}
		\limsup_{t\to\infty} \rho\left(\mu_t,\nu_t\right) \leq \rho\left(\mu,\nu\right).
		\end{equation*}
	\end{theorema}
	As mentioned before, Theorem \ref{thm:weak_star_measures_bound_affiliation} originally appeared in \cite{coquet_kamae_menfrance} for the case $d=1$ and $\mathcal{L}= \mathbb{Z}$. For the convenience of the reader we provide a proof in appendix \ref{appendix:proof_of_theorem_weak}. In fact, Lemma \ref{lemma:upper_bound_for_mix_spectral_measures_ac} is a special case of the following inequality. 
	\begin{lemma}
		\label{lemma:upper_bound_for_mix}
		Let $u,v$ be Wiener sequences and let $\mu_{u},\mu_{v}$ be their spectral measures. Then
		\[
		\limsup_{R\to\infty} \left|\frac{1}{m_d(B_R)}\sum_{n\in\mathcal{L}\cap B_R} u(n) \overline{v(n)} \right| \le \rho\left(\mu_{u},\mu_{v}\right).
		\]
		In particular, if $\mu_u$ and $\mu_v$ are mutually singular, then
		\[
		\lim_{R\to\infty}\frac{1}{m_d(B_R)}\sum_{n\in\mathcal{L}\cap B_R}u(n) \overline{v(n)} = 0.
		\]
	\end{lemma}
	\begin{proof}
		Consider the family of measures
		\begin{equation*}
		{\rm d}\mu_{u}^{R}(x):= \frac{1}{m_d\left(B_{R}\right)}\left|\sum_{n\in\mathcal{L}\cap B_R}u(n)e\left(\inner{n}{x}\right) \right|^2 {\rm d}m_d(x), \qquad R\ge 1.
		\end{equation*}
		We will first show that $\mu_{u}^R \xrightarrow{\ w^\ast\ } \mu_u$ as $R\to\infty$. We do so by examining the Fourier coefficients. For any $k\in\mathcal{L}$ 
		\begin{align*}
		\widehat{\mu_{u}^{R}}(k) &= \frac{1}{m_d\left(B_{R}\right)} \int_{\mathcal{D}}\left|\sum_{n\in\mathcal{L}\cap B_R}u(n)e\left(\inner{n}{x}\right) \right|^2 \overline{e\left(\inner{k}{x}\right)} {\rm d}m_d(x) \\ &= \frac{1}{m_d\left(B_{R}\right)} \int_{\mathcal{D}}\left(\sum_{n\in\mathcal{L}\cap B_R}u(n)e\left(\inner{n}{x}\right)\right)\left(\sum_{n^\prime\in\mathcal{L}\cap B_R}\overline{u(n^\prime)}\overline{e\left(\inner{n^\prime+k}{x}\right)}\right) {\rm d}m_d(x) \\ &= \frac{1}{m_d\left(B_{R}\right)} \sum_{n,n^{\prime}\in\mathcal{L}\cap B_R} u(n)\overline{u(n^{\prime})} \left(\int_{\mathcal{D}} e\left(\inner{x}{n-n^\prime-k}\right){\rm d}m_d(x) \right) \\ &= \frac{1}{m_d\left(B_{R}\right)} \sum_{n\in\mathcal{L}\cap B_R} u(n)\overline{u(n+k)}.
		\end{align*}
		Since $u$ is a Wiener sequence relation (\ref{eq:correlation_of_seq_u}) yields that
		\[
		\lim_{R\to\infty} \widehat{\mu_{u}^{R}}(k) = \lim_{R\to\infty} \frac{1}{m_d\left(B_{R}\right)} \sum_{n\in\mathcal{L}\cap B_R} u(n)\overline{u(n+k)} = \widehat{\mu_{u}}(k), \qquad \text{for all} \ k\in\mathcal{L}.
		\]
		Pointwise convergence of the Fourier coefficients implies that $\mu_u^{R}\xrightarrow{\ w^\ast\ } \mu_u$ as $R\to \infty$. Symmetrically we define 
		\[
		{\rm d}\mu_{v}^{R}(x):= \frac{1}{m_d\left(B_{R}\right)}\left|\sum_{n\in\mathcal{L}\cap B_R}v_{n}e\left(\inner{n}{x}\right) \right|^2  {\rm d}m_d(x)
		\]
		and obtain $\mu_{v}^R\xrightarrow{\ w^\ast\ } \mu_v$ as $R\to\infty$. Now, Theorem \ref{thm:weak_star_measures_bound_affiliation} implies that
		\[
		\limsup_{R\to\infty} \rho\left(\mu_u^R,\mu_v^R\right) \leq \rho\left(\mu_u,\mu_v\right).
		\] 
		By using ${\rm d} \sigma = {\rm d} m_d$ in the Hellinger integral (\ref{eq:affinity_of_measures}), it remains to apply the triangle inequality and observe that
		\begin{align*}
		\rho\left(\mu_u^R,\mu_v^R\right) &= \frac{1}{m_d\left(B_{R}\right)} \int_{\mathcal{D}} \left|\sum_{n\in\mathcal{L}\cap B_R}u(n)e\left(\inner{n}{x}\right) \right|\left|\sum_{n^{\prime}\in\mathcal{L}\cap B_R}v(n^\prime)e\left(\inner{n^\prime}{x}\right) \right| {\rm d}m_d(x) \\ &\geq \frac{1}{m_d\left(B_{R}\right)} \left|\sum_{n,n^{\prime}\in\mathcal{L}\cap B_R} \int_{\mathcal{D}} u(n)\overline{v(n^\prime)} e\left(\inner{x}{n-n^{\prime}} \right) {\rm d}m_d(x) \right|\\ & = \frac{1}{m_d\left(B_{R}\right)} \left|\sum_{n\in\mathcal{L}\cap B_R} u(n) \overline{v(n)} \right|
		\end{align*}
		which gives the result.
	\end{proof}
	
	\section{Proof of Theorem \ref{thm:weak_star_measures_bound_affiliation}}
	\label{appendix:proof_of_theorem_weak}
	
	The proof we present here is similar to the one presented in \cite{coquet_kamae_menfrance}, except for minor straightforward modifications. 
	\begin{proof}[Proof of Theorem \ref{thm:weak_star_measures_bound_affiliation}]
		We fix a reference measure $\sigma$ so that $\mu\ll\sigma$ and $\nu\ll \sigma$ and also fix representatives of ${\rm d}\mu/{\rm d}\sigma$ and ${\rm d}\nu/{\rm d}\sigma$. Let $\epsilon>0$ and consider
		\begin{equation*}
		A:= \left\{x\in \mathcal{D} \mathrel{\Big|} \frac{{\rm d}\mu}{{\rm d}\sigma}(x) = 0 \right\}, \qquad B:= \left\{x\in \mathcal{D} \mathrel{\Big|} \frac{{\rm d}\nu}{{\rm d}\sigma}(x) = 0 \right\}\setminus A
		\end{equation*}
		and
		\[
		V_j = V_j(\epsilon) :=  \left\{x\in \mathcal{D} \mathrel{\Big|} (1+\epsilon)^{j-1}\frac{{\rm d}\mu}{{\rm d}\sigma}(x) \leq  \frac{{\rm d}\nu}{{\rm d}\sigma}(x) < (1+\epsilon)^j\frac{{\rm d}\mu}{{\rm d}\sigma}(x) \right\}\setminus \left(A\cup B\right) 
		\]
		for $j\in\mathbb{Z}$. Integrating with respect to ${\rm d}\sigma$ gives 
		\begin{equation}
		\label{eq:inequalities_of_partition}
		(1+\epsilon)^{j-1}\mu\left(V_j\right) \leq \nu\left(V_j\right) \leq (1+\epsilon)^j\mu\left(V_j\right)
		\end{equation}
		for all $j\in\mathbb{Z}$. The collection $\left\{A,B,\left\{V_j\right\}_{j\in\mathbb{Z}}\right\}$ forms a partition of $\mathcal{D}$. Since $\mu\left(\mathcal{D}\right)<\infty$, we may fix $N=N(\epsilon)$ large enough so that
		\[
		\sum_{|j|\geq N} \mu\left(V_j\right) \leq \epsilon^2.
		\]
		With $C:= \bigcup_{|j|\geq N} V_j$, the collection of sets
		$$
		\left\{A,B,C,V_{-N+1},\ldots,V_0,\ldots,V_{N-1}\right\} =: \left\{U_{1},\ldots, U_{2(N+1)}\right\}
		$$
		forms a finite partition of $\mathcal{D}$. Notice that $\mu(U_1) = \nu(U_2) = 0$ and that $\mu(U_3) \leq \epsilon^2$. By outer-regularity of the measures $\mu$ and $\nu$, we may choose open sets $\left\{O_j\right\}_{j=1}^{2(N+1)}$ so that $U_{j}\subset O_j$, $$\max\left\{\mu(O_1), \nu(O_2), \mu(O_3) \right\} \leq \epsilon,$$ and 
		\[
		(1+\epsilon)^{1/2}\mu\left(U_j\right)\geq \mu\left(O_j\right), \quad (1+\epsilon)^{1/2}\nu\left(U_j\right)\geq \nu\left(O_j\right), \qquad \text{for } j\geq 4.
		\]  
		Now, let $\left(f_j\right)_{j=1}^{2(N+1)}$ be a continuous partition of unity subordinated to the open covering $O_j$ of $\mathcal{D}$. For $j\ge 4$ we have
		\begin{equation}
		\label{eq:partition_of_unity_inq_1}
		\int_{\mathcal{D}} f_j {\rm d} \mu \leq \mu\left(O_j\right) \leq (1+\epsilon)^{1/2}\mu\left(U_j\right)
		\end{equation} 
		and 
		\begin{equation}
		\label{eq:partition_of_unity_inq_2}
		\int_{\mathcal{D}} f_j {\rm d} \nu \leq \nu\left(O_j\right)\leq (1+\epsilon)^{1/2}\nu\left(U_j\right).
		\end{equation}
		Furthermore, for $j=1,2,3$ we have
		\begin{equation}
		\label{eq:partition_of_unity_inq_special_sets}
		\left(\int_{\mathcal{D}} f_j {\rm d} \mu \right) \left(\int_{\mathcal{D}} f_j {\rm d} \nu \right)  \leq \epsilon\max\{\mu(\mathcal{D}),\nu(\mathcal{D})\}.
		\end{equation}
		Let $\sigma_t:= \mu_t + \nu_t$. By the Cauchy-Schwarz inequality
		\begin{align*}
		\rho\left(\mu_t,\nu_{t}\right) &= \int_{\mathcal{D}} \left(\frac{{\rm d}\mu_t}{{\rm d}\sigma_t}\right)^{1/2}\left(\frac{{\rm d}\nu_t}{{\rm d}\sigma_t}\right)^{1/2} {\rm d}\sigma_t \\ &= \sum_{j=1}^{2(N+1)}\int_{\mathcal{D}} f_j\left(\frac{{\rm d}\mu_t}{{\rm d}\sigma_t}\right)^{1/2}\left(\frac{{\rm d}\nu_t}{{\rm d}\sigma_t}\right)^{1/2} {\rm d}\sigma_t \\ & \leq \sum_{j=1}^{2(N+1)} \left(\int_{\mathcal{D}}f_j\frac{{\rm d}\mu_t}{{\rm d}\sigma_t}{\rm d}\sigma_t\right)^{1/2}\left(\int_{\mathcal{D}}f_j\frac{{\rm d}\nu_t}{{\rm d}\sigma_t}{\rm d}\sigma_t\right)^{1/2} \\ &= \sum_{j=1}^{2(N+1)} \left(\int_{\mathcal{D}}f_j{\rm d}\mu_t\right)^{1/2}\left(\int_{\mathcal{D}}f_j{\rm d}\nu_t\right)^{1/2}.
		\end{align*}
		Therefore, by the weak-star convergence assumption combined with relations (\ref{eq:partition_of_unity_inq_1}), (\ref{eq:partition_of_unity_inq_2}) and (\ref{eq:partition_of_unity_inq_special_sets}) we see that
		\begin{align*}
		\limsup_{t\to\infty} \rho\left(\mu_t,\nu_{t}\right) & \leq \lim_{t\to\infty} \sum_{j=1}^{2(N+1)} \left(\int_{\mathcal{D}}f_j{\rm d}\mu_t\right)^{1/2}\left(\int_{\mathcal{D}}f_j{\rm d}\nu_t\right)^{1/2} \\ & =\sum_{j=1}^{2(N+1)} \left(\int_{\mathcal{D}}f_j{\rm d}\mu\right)^{1/2}\left(\int_{\mathcal{D}}f_j{\rm d}\nu\right)^{1/2} \\ & \leq (1+\epsilon)^{1/2} \sum_{j=4}^{2(N+1)} \sqrt{\mu\left(U_j\right)\nu\left(U_j\right)} + 3\sqrt{\max\{\mu(\mathcal{D}),\nu(\mathcal{D})\}\epsilon}.
		\end{align*}
		It remains to bound the sum on the right hand side. Using (\ref{eq:inequalities_of_partition}) and the definition of $V_j$ we see that
		\begin{align*}
		\sum_{j=4}^{2(N+1)} \sqrt{\mu\left(U_j\right)\nu\left(U_j\right)}  &= \sum_{j=-N+1}^{N-1} \sqrt{\mu\left(V_j\right)\nu\left(V_j\right)} \\ & \stackrel{(\ref{eq:inequalities_of_partition})}{\leq} \sum_{j=-N+1}^{N-1} (1+\epsilon)^{j/2}\mu\left(V_j\right)  \\ & \leq (1+\epsilon)^{1/2}\sum_{j=-N+1}^{N-1} \left\{\int_{V_{j}} \left(\frac{{\rm d}\mu}{{\rm d}\sigma}\right)^{1/2}\left(\frac{{\rm d}\nu}{{\rm d}\sigma}\right)^{1/2} {\rm d}\sigma \right\} \\ &\leq (1+\epsilon)^{1/2}\rho\left(\mu,\nu\right).
		\end{align*}
		Altogether
		\[
		\limsup_{t\to\infty} \rho\left(\mu_t,\nu_t\right) \leq (1+\epsilon)\rho\left(\mu,\nu\right) + 3\sqrt{\max\{\mu(\mathcal{D}),\nu(\mathcal{D})\}\epsilon}
		\]
		and by taking $\epsilon\to0$ we are done.
	\end{proof}

	
\end{document}